\documentclass{amsart}
\usepackage{txfonts}
\usepackage{}

\newtheorem{theorem}{Theorem}[section]

\newtheorem{corollary}[theorem]{Corollary}

\theoremstyle{definition}

\theoremstyle{remark}
\newtheorem{remark}[theorem]{Remark}

\numberwithin{equation}{section}

\begin{document}

\title{Four-manifolds with  postive Yamabe constant}

\author{Hai-Ping Fu}
\address{Department of Mathematics,  Nanchang University, Nanchang, P.
R. China 330031}
\email{mathfu@126.com}
\thanks{Supported in part by National Natural Science Foundations of China \#11261038 and \#11361041.}






\begin{abstract}
We refine Theorem B due to Gursky \cite{G3}. As applications, we give some rigidity theorems on four-manifolds with  postive Yamabe constant. In particular, these rigidity theorems are sharp for our conditions have the additional properties of being sharp. By this we mean that we can precisely characterize the case of equality.
 We prove some classification theorems of four manifolds according to some conformal invariants (see Theorems 5.1, 1.3 and 1.6), which reprove and generalize  the conformally invariant sphere theorem of Chang-Gursky-Yang \cite{CGY}, i.e., Theorem D.
\end{abstract}

\maketitle

\section{Introduction and main results}
In \cite {F}, the author proved that an  $n$-manifold with harmonic curvature  satisfied some inequality that the upper bound of some
curvature functionals  is controlled by  Yamabe constant is actually isometric to a quotient of the standard sphere or an Einstein manifold. By this we mean that we can precisely characterize the case of equality. The aim of this paper is to present some rigidity results in the subject of curvature pinching on four-manifolds with   postive Yamabe constant.

Let $(M^n, g)$ be an $n$-dimensional Riemannian manifold. The
decomposition of the Riemannian curvature tensor $Rm$  into irreducible components yield
\begin{eqnarray*}
Rm=W+\frac{1}{n-2}\mathring{Ric}\circledwedge g+
\frac{R}{2n(n-1)}g\circledwedge g,
\end{eqnarray*}
where  $W$, $Ric$, $\mathring{Ric}=Ric-\frac{R}{n}g$ and $R$ denote the Weyl curvature tensor,  Ricci tensor,  trace-free Ricci tensor   and   scalar curvature, respectively. When the divergence of the Weyl curvature tensor  $W$ is vanishing, i.e., $\delta W=0$, $(M^n, g)$ is said to be a manifold with  harmonic Weyl tensor.

The sphere theorem for $\frac14$-pinched Riemannian manifolds, conjectured by H. Rauch
in 1951, is a  good example of the deep connections between the topology
and the geometry of Riemannian manifolds. After a large amount of research, now we know that the
answer is positive, due to the fundamental work of Klingenberg, Berger and Rauch
for the topological statement and the recent proof of the original conjecture by Brendle
and Schoen \cite{BS}, based on the results of B\"{o}hm and Wilking \cite{BW}.

Now, since the work of Huisken \cite{Hu} and Margerin \cite{M}, we know that there exists a positive-dimensional constant
$C(n)$ such that if $|W+\frac{1}{n-2}\mathring{Ric}\circledwedge g|^2 < C(n)R^2(g)$, then $M^n$ is diffeomorphic to a quotient of the standard unit sphere. In particular, Margerin improved the constant in dimension four, and obtained the
optimal theorem in \cite{M}.

The common feature of all the above results is to give topological information on a manifold
that carry a metric whose curvature satisfies a certain pinching at each point.
The question one raises here is whether one can characterize the topology and the geometry of Riemannian manifolds by means of integral
pinching conditions instead of pointwise ones.
Some results in this direction on four manifolds were obtained by some authors \cite{{BC},{CGY},{CZ},{G2},{G3},{HV}}.

In four-manifolds, the Weyl functional $\int |W_{g}|^2$  has long been an object of
interest to physicists. Suppose $M^4$ is a four-dimensional manifold. Then the Hodge
$\ast$-operator induces a splitting of the space of two-forms $\wedge^2=\wedge^2_{+}+\wedge^2_{-}$ into the
subspace of self-dual forms $\wedge^2_{+}$ and anti-self-dual forms $\wedge^2_{-}$. This splitting in
turn induces a decomposition of the Weyl curvature  into its self-dual and anti-self-dual components $W^{\pm}$. A four-manifold is said to be self-dual (resp., anti-self-dual) if $W^-=0$ (resp., $W^+=0$). It is said to be a manifold with half harmonic Weyl tensor if $\delta W^{\pm}=0$. By
the Hirzebruch signature formula (see \cite{B}),
\begin{equation}\int_{M}|W^{+}|^{2}-\int_{M}|W^{-}|^{2}= 48\pi^2\sigma(M),\end{equation}
where $\sigma(M)$ denotes the signature of $M$.
A consequence of (1.1) is that the study of the Weyl functional is completely equivalent to
the study of the self-dual Weyl functional $\int_{M}|W^{+}|^{2}$. Gursky \cite{{G2},{G3}} have obtained some good and interesting results by studying $\int_{M}|W^{+}|^{2}$ (See Theorems A, B and C). For background material on this
condition we recommend chapter 16 of \cite{B} and the paper of Derdzinski \cite{D}.

Our formulation of some results will be given in term of the Yamabe invariant.
Now we  introduce the definition of the Yamabe constant. Given a compact  Riemannian $n$-manifold $M$, we consider the Yamabe functional
$$Q_g\colon C^{\infty}_{+}(M)\rightarrow\mathbb{R}\colon f\mapsto Q_g(f)=\frac{\frac{4(n-1)}{n-2}\int_M|\nabla f|^2\mathrm{d}v_g+\int_M Rf^2\mathrm{d}v_g}{(\int_M f^{\frac{2n}{n-2}}\mathrm{d}v_g)^{\frac{n-2}{n}}},$$
where $R$ denotes the   scalar curvature of $M$.
It follows that $Q_g$ is bounded below by H\"{o}lder inequality. We set
$$\mu([g])=\inf\{Q_g(f)|f\in C^{\infty}_{+}(M)\}.$$
This constant $\mu([g])$ is an invariant of the conformal class of $(M, g)$, called the Yamabe constant.
The important works of Schoen, Trudinger and Yamabe showed that the infimum in the above is always achieved (see \cite{{A},{LP}}).
The Yamabe constant of a given compact manifold
is determined by the sign of scalar curvature \cite{A}.

Gursky \cite{{G2},{G3}} proved the  three striking Theorems A, B and C. As byproducts, Gursky \cite{{G2},{G3}} obtained these integral pinching results, which are generalizations of the Bochner theorem
in dimension four (See Theorems E, F and G).

\noindent
{ {\bf Theorem A (Gursky \cite{G2}).} Let $(M^4, g)$  be a $4$-dimensional   compact Riemannian manifold  with  positive Yamabe constant  and the space of self-dual harmonic two-forms $H^2_{+}(M^4)\neq0$.
Then
 \begin{equation*} \int_M|W^{+}|^{2}\geq \frac{16}{3}{\pi}^2(2\chi(M^4)+3\sigma(M^4)),\end{equation*}
 where $\chi(M)$ is the Euler-Poincar\'{e} characteristic of $M$.
 Furthermore, equality holds in the above inequality if and only if
$g$ is conformal to a positive K\"{a}hler-Einstein metric.}

\noindent
{{\bf Theorem B (Gursky \cite{G3}).} Let $(M^4, g)$  be a $4$-dimensional   compact Riemannian manifold  with  positive Yamabe constant  and $\delta W^{+}=0$.
Then either $(M^4, g)$
is anti-self-dual, or
 \begin{equation*} \int_M|W^{+}|^{2}\geq \frac{16}{3}{\pi}^2(2\chi(M^4)+3\sigma(M^4)).\end{equation*}
 Furthermore, equality holds in the above inequality if and only if
$g$  is a positive Einstein metric
which is either K\"{a}hler, or the quotient of a K\"{a}hler manifold by a free, isometric,
anti-holomorphic involution.}

\noindent
{ {\bf Theorem C (Gursky \cite{G2}).} Let $(M^4, g)$  be a $4$-dimensional   compact Riemannian manifold  with  positive Yamabe constant  and the space of harmonic $1$-forms $H^1(M^4)\neq0$.
Then
 \begin{equation*} \int_M|W^{+}|^{2}\geq 8{\pi}^2(2\chi(M^4)+3\sigma(M^4)).\end{equation*}
 Furthermore, equality holds in the above inequality if and only if
$(M^4, g)$ is conformal to a quotient of $\mathbb{R}^1\times \mathbb{S}^{3}$ with the product metric.}

\noindent
Chang, Gursky and Yang \cite{CGY} proved that a  four manifold with positive Yamabe constant which satisfies the strict inequality for the Weyl functional $\int |W_{g}|^2$ is actually diffeomorphic to a quotient of the sphere and precisely
characterize the case of equality. We state this result of Chang-Gursky-Yang as follows:

\noindent
{{\bf Theorem D(Chang-Gursky-Yang \cite{CGY}).} Let $(M^4, g)$  be a $4$-dimensional   compact Riemannian manifold  with  positive Yamabe constant. If
 \begin{equation*} \int_M|W|^{2}\leq 16{\pi}^2\chi(M),\end{equation*} then  1) $M^4$  is diffeomorphic to the round $\mathbb{S}^4$ or to the real projective space $\mathbb{RP}^4$;

2) $M^4$ is conformal to a manifolds which is isometrically covered by $\mathbb{S}^1\times \mathbb{S}^{3}$ with the product metric;

3) $M^4$ is conformal to $\mathbb{CP}^2$ with the Fubini-Study metric.}

\noindent
Bour and Carron \cite{BC}  reprove and extend to higher degrees and higher dimensions
Theorems F and G obtained by M. Gursky. Bour \cite{B2}  give a new proof of  Theorem D
under a stronger pinching assumption, which is entirely based on the study of a geometric
flow, and doesn't rely on the pointwise version of the theorem, due to Margerin. Chen and Zhu \cite{CZ} proved a classification theorem of $4$-manifolds
according to some conformal invariants, which generalizes the conformally
invariant sphere theorem of Chang-Gursky-Yang \cite{CGY}, i.e., Theorem D
under the strict inequality assumption, and relies on Chen-Tang-Zhu's classification on four-manifolds with positive isotropic curvature \cite{CTZ}.

In this note, we refine Theorems B and E due to Gursky, and obtain Theorem 1.1 as follows:
\begin{theorem}
Let $(M^4, g)$ be a  $4$-dimensional   compact Riemannian manifold with
$\delta W^{\pm}=0$  and positive Yamabe constant $\mu([g])$. If
\begin{equation}\int_{M}|W_g^{\pm}|^{2}=\frac 16\mu^2([g]).\end{equation}
Then  $\nabla W^{\pm}=0$ and $W^{\pm}$ has exactly two distinct eigenvalues at each point. Hence $(M^4, g)$ is a K\"{a}hler manifold of positive constant scalar curvature.
\end{theorem}

Combing some results due to Gursky and  Chen-Tang-Zhu's classification on four-manifolds with positive isotropic curvature with Theorem 1.1, we give some rigidity theorems on four-manifolds with positive Yamabe constant.

\begin{theorem}
Let $(M^4, g)$  be a $4$-dimensional   compact Riemannian manifold  with harmonic Weyl tensor and   positive Yamabe constant. If
 \begin{equation*}\int_M|W|^{2}=\frac{1}{6}{\mu^2([g])},\end{equation*} then  $(M^4, {g})$
 is  a $\mathbb{CP}^2$ with the Fubini-Study metric.
\end{theorem}

\begin{theorem}
Let $(M^4, g)$  be a $4$-dimensional   compact Riemannian manifold  with   positive Yamabe constant. If
 \begin{equation*}\int_M|W|^{2}\leq\frac{1}{6}{\mu^2([g])},\end{equation*} then  1) $\tilde{g}$ is a Yamabe minimizer and $(M^4, \tilde{g})$  is a $\mathbb{CP}^2$ with the Fubini-Study metric;

 2) $(M^4, g)$ is
diffeomorphic to $\mathbb{S}^4,$ $\mathbb{RP}^4$, $\mathbb{S}^3\times \mathbb{R}/G$ or a connected sum of them. Here
$G$ is a cocompact fixed point free discrete subgroup of the isometry group
of the standard metric on $\mathbb{S}^3\times \mathbb{R}$.
\end{theorem}

\begin{theorem}
Let $(M^4, g)$  be a $4$-dimensional   compact Riemannian manifold  with harmonic Weyl tensor and   positive Yamabe constant. If
 \begin{equation*}\int_M|W|^{2}<\frac{64}{3}\pi^2\chi(M).\end{equation*}
  Then
 1) $\tilde{g}$ is a Yamabe minimizer and $(M^4, \tilde{g})$  is  the stand sphere $\mathbb{S}^4$ or the real projective space $\mathbb{RP}^4$;

2) $(M^4, {g})$
 is  a $\mathbb{CP}^2$ with the Fubini-Study metric.
\end{theorem}

\begin{theorem}
Let $(M^4, g)$  be a $4$-dimensional   compact Riemannian manifold  with harmonic Weyl tensor and   positive Yamabe constant. If
 \begin{equation*}\int_M|W|^{2}=\frac{64}{3}\pi^2\chi(M),\end{equation*} then
 1) $\tilde{g}$ is a Yamabe minimizer and $(M^4, \tilde{g})$  is the manifold which is isometrically covered by $\mathbb{S}^1\times \mathbb{S}^{3}$ with the product metric
or the manifold which is isometrically covered by $\mathbb{S}^1\times \mathbb{S}^{3}$ with a rotationally symmetric Derdzi\'{n}ski metric  (see \cite{{Ca},{D1}});

 2)  $(M^4, {g})$ is isometric to a quotient of $\mathbb{S}^2\times \mathbb{S}^{2}$ with the product metric.
\end{theorem}

\begin{theorem}
Let $(M^4, g)$  be a $4$-dimensional   compact Riemannian manifold which is not diffeomorphic to $\mathbb{S}^4$ or $\mathbb{RP}^4$  with  positive Yamabe constant. If
 \begin{equation*}16\pi^2\chi(M)<\int_M|W|^{2}\leq\frac{64}{3}\pi^2\chi(M),\end{equation*} then
 1) $\tilde{g}$ is a Yamabe minimizer and $(M^4, \tilde{g})$ is isometric to a quotient of  $\mathbb{S}^2\times \mathbb{S}^{2}$ with the product metric;

 2) $(M, g)$ has $\chi(M)=3$,  $b_1=0$ and $b_2=1$, where $b_i$ denotes the $i$-th Betti number of $M$, and has not harmonic Weyl tensor.
\end{theorem}

\begin{theorem}
Let $(M^4, g)$  be a   $4$-dimensional   compact   Riemannian manifold  with harmonic curvature and   positive  scalar curvature. If
\begin{equation}\int_M|W|^{2}+4\int_M|\mathring{Ric}|^2
=\frac{1}{3}\mu^2([g]),\end{equation}
then 1) $M^4$ is a quotient of $\mathbb{S}^2\times \mathbb{S}^{2}$ with the product metric;

2) $M^4$ is covered isometrically by $\mathbb{S}^1\times \mathbb{S}^{3}$ with the product metric;

3) $M^4$ is covered isometrically by $\mathbb{S}^1\times \mathbb{S}^{3}$ with a rotationally symmetric Derdzi\'{n}ski metric.
\end{theorem}

\begin{theorem}
Let $(M^4, g)$  be a   $4$-dimensional   compact  Riemannian manifold  with harmonic curvature and   positive  scalar curvature. If
\begin{equation}\int_M|W|^{2}+4\int_M|\mathring{Ric}|^2
<\frac{1}{3}\mu^2([g]),\end{equation}
then 1) $M^4$ is a quotient of the round $\mathbb{S}^4$;

2) $M^4$ is a $\mathbb{CP}^2$ with the Fubini-Study metric.
\end{theorem}
\begin{corollary}
Let $(M^4, g)$  be a $4$-dimensional   complete Einstein manifold  with   positive scalar curvature. If
\begin{equation*}\int_M|W|^{2}\leq\frac{1}{3}\mu^2([g]),\end{equation*}
then 1) $(M^4, g)$ is isometric to either  a stand sphere $\mathbb{S}^4$ or  a real projective space $\mathbb{RP}^4$;

2) $(M^4, g)$ is isometric to a  $\mathbb{CP}^2$ with the Fubini-Study metric;

3) $(M^4, g)$ is isometric to a quotient of
$\mathbb{S}^2\times \mathbb{S}^{2}$ with the product metric.
\end{corollary}

\begin{remark}
For Riemannian manifolds  with harmonic curvature and  dimensions $n\geq4$, the author  proved some similar results in \cite{F}.
\end{remark}

\begin{theorem}
Let $(M^4, g)$  be a $4$-dimensional   compact Riemannian manifold  with harmonic curvature and   positive scalar curvature. If
 \begin{equation*}\int_M|W|^{2}\leq\frac{1}{3}\mu^2([g]),\end{equation*}
then
1) $(M^4, {g})$  is conformally flat with positive constant scalar curvature;

2) $(M^4, {g})$  is a $\mathbb{CP}^2$ with the Fubini-Study metric;

3) $(M^4, g)$ is isometric to a quotient of
$\mathbb{S}^2\times \mathbb{S}^{2}$ with the product metric.
\end{theorem}

{Acknowledgement:} The author would like to thank  Professor M. J. Gursky for
some helpful suggestions. The author  is very grateful for Professors Haizhong Li, Kefeng liu and Hongwei Xu's encouragement and help.


\section{Four manifolds with half harmonic Weyl tensor}
In order to prove some results in this article, we need the following Weyl Estimate proved by Gursky in \cite{G3}.

\noindent
{{\bf Theorem E (Gursky \cite{G3}).}
Let $(M^4, g)$ be a  $4$-dimensional   compact Riemannian manifold with
$\delta W^{\pm}=0$  and positive Yamabe constant $\mu([g])$. If $$\int_{M}|W_g^{\pm}|^{2}< \frac 16\mu^2([g]) ,$$
 then $(M^4, g)$ is anti-self-dual (resp., self-dual).
}
\begin{remark}
First, Gursky obtained   an improved Kato inequality $|\nabla W^{+}|^2\geq \frac{5}{3}|\nabla |W^{+}||^2$ in \cite{G3}. Thus using the Bochner technique,  Gursky proved Theorems B, E, F and G by introducing the corresponding functional and conformal invariant
with the modified scalar curvature $R-\sqrt{6}|W^{\pm}$ in \cite{G3}. Based on Gursky's improved Kato inequality,  we can reprove Theorems B, E, F and G only by using the modified Bochner technique (See \cite{{BC},{F},{FL}}). In order to prove Theorem 1.1, we also need the following proof of Theorem E.
\end{remark}
\begin{proof}First, we recall the following
Weitzenb\"{o}ck formula (see \cite{B} and \cite{B2})
\begin{eqnarray}
\triangle|W^{\pm}|^2=2|\nabla W^{\pm}|^2+R|W^{\pm}|^2-144det_{\wedge^2_\pm}W^{\pm}.
\end{eqnarray}
From (2.1), by the Kato inequality $|\nabla W^{+}|^2\geq \frac{5}{3}|\nabla |W^{+}||^2$ (\cite{G3}), we obtain
\begin{equation} |W^{\pm}|\triangle|W^{\pm}|\geq\frac{2}{3}|\nabla |W^{\pm}||^2+\frac{1}{2}R|W^{\pm}|^2-72det_{\wedge^2_\pm}W^{\pm}.\end{equation}
By a simple Lagrange multiplier argument it is easily verified that
\begin{equation}-144det_{\wedge^2_\pm}W^{\pm}\geq-\sqrt{6}|W^{\pm}|^3\end{equation}
and equality is attained at a point where $W^{\pm}\neq0 $ if and only if $W^{\pm}$ has precisely
two eigenvalues. By (2.2) and (2.3), we get
\begin{equation} |W^{\pm}|\triangle|W^{\pm}|\geq\frac{2}{3}|\nabla |W^{\pm}||^2+\frac{1}{2}R|W^{\pm}|^2-\frac{\sqrt{6}}{2}|W^{\pm}|^3.\end{equation}
Let $u=|W^{\pm}|$.
By (2.4), we compute
\begin{equation}
\begin{split}
u^{\alpha}\triangle u^{\alpha}&=u^{\alpha}\left(\alpha(\alpha-1)u^{\alpha-2}|\nabla u|^2+\alpha u^{\alpha-1}\triangle u\right)\\
&=\frac{\alpha-1}{\alpha}|\nabla u^{\alpha}|^2+\alpha
u^{2\alpha-2}u\triangle u\\
&\geq(1-\frac{1}{3\alpha})|\nabla u^{\alpha}|^2
-\frac{\sqrt{6}}{2}\alpha u^{2\alpha+1}+\frac{R\alpha }{2} u^{2\alpha},
\end{split}
\end{equation}
where $\alpha$ is a positive constant.
Integrating (2.5) by parts, we get
\begin{equation}
\begin{split}
(2-\frac{1}{3\alpha})\int_{M}|\nabla u^\alpha|^2-\frac{\sqrt{6}}{2}\alpha\int_{M}u^{2\alpha+1}
+\frac{\alpha }{2}\int_{M}Ru^{2\alpha}\leq 0.
\end{split}
\end{equation}
By the H\"{o}lder inequality and (2.6), we have
\begin{equation}
\begin{split}
(2-\frac{1}{3\alpha})\int_{M}|\nabla u^\alpha|^2-\frac{\sqrt{6}}{2}\alpha\left(\int_{M}u^{4\alpha}\right)^{\frac{1}{2}}\left(\int_{M}u^{2}\right)^{\frac{1}{2}}+\frac{\alpha }{2}\int_{M}Ru^{2\alpha}\leq 0.
\end{split}
\end{equation}
For $2-\frac{1}{3\alpha}>0$, by the definition of Yamabe constant and (2.7), we get
\begin{equation}
\begin{split}
0\geq(2-\frac{1}{3\alpha})\frac16\mu([g])\left(\int_M  u^{4\alpha}\right)^{\frac{1}{2}}-\frac{\sqrt{6}}{2}\alpha\left(\int_{M}u^{4\alpha}\right)^{\frac{1}{2}}\left(\int_{M}u^{2}\right)^{\frac{1}{2}}\\
+\frac {9\alpha^2-6\alpha+1}{18\alpha}\int_M Ru^{2\alpha}.
\end{split}
\end{equation}
We choose $\alpha=\frac 13$, from (2.8) we get
\begin{equation}
\begin{split}
0\geq\left[\frac{1}{\sqrt{6}}\mu([g])-\left(\int_{M}u^{2}\right)^{\frac{1}{2}}\right]\left(\int_M  u^{\frac43}\right)^{\frac{1}{2}}.
\end{split}
\end{equation}
We choose $\left(\int_{M}|W^{\pm}|^{2}\right)< \frac{1}{6}\mu^2([g])$ such that the above inequality imply $\int_M  u^{\frac43}=0$, that is, $W^{\pm}=0$, i.e., $(M^4, g)$ is anti-self-dual, or self-dual.
\end{proof}
\begin{remark}
For $0\leq k\leq\frac n2$, by the Kato inequality for harmonic $k$-form $\omega$ (see \cite{Bo})
$\frac{n+1-k}{n-k}|\nabla|\omega||\leq|\nabla \omega|$ and the two Weitzenb\"{o}ck formulas in \cite{G2},
one has
$$\frac 12\triangle|\omega|^2\geq|\nabla\omega|^2-\frac{\sqrt{6}}{3}|W^{\pm}||\omega|^2+\frac 13 R|\omega|^2\geq\frac32|\nabla|\omega||^2-\frac{\sqrt{6}}{3}|W^{\pm}||\omega|^2+\frac 13 R|\omega|^2,\\ \forall \omega\in H^2_{\pm}(M^4) $$
and
$$\frac 12\triangle|\omega|^2\geq\frac43|\nabla|\omega||^2-\frac{\sqrt{3}}{2}|\mathring{Ric}||\omega|^2+\frac 14 R|\omega|^2,\\ \forall \omega \in H^1(M^4). $$
Based on the above two Weitzenb\"{o}ck formulas, using the same argument as in the proof of Theorem E, we can obtain two results of Gursky \cite{{G2},{G3}} as follows:

\noindent
{{\bf Theorem F(Gursky \cite{{G2},{G3}}).} Let $(M^4, g)$ be a  $4$-dimensional   compact Riemannian manifold with
 positive Yamabe constant $\mu([g])$.

 i) If $$\int_{M}|W^{\pm}|^{2}< \frac 16\mu^2([g]) ,$$
 then $H^2_{\pm}(M^4)=0$ and $b^{\pm}_2(M)=0$;

 ii) If $$\int_{M}|\mathring{Ric}|^{2}< \frac {1}{12}\mu^2([g]) ,$$
 then $H^1(M^4)=0$ and $b_1(M)=0$.}
 \end{remark}

\begin{proof}[{\bf Proof of Theorem
1.1}]
(1.2) implies that the equality holds in (2.9).
When the equality holds in (2.9), all inequalities leading to (2.7)
become equalities. From (2.8), the function $u^\alpha$ attains the infimum in the Yamabe functional. From (2.7), the equality for the H\"{o}lder inequality implies that $u$ is constant, i.e., $|W^{\pm}|$ is constant. Hence at every point,  it has an eigenvalue of
multiplicity $2$ and another of multiplicity $1$, i.e., $W^{\pm}$ has eigenvalues $\{-\frac{R}{12},-\frac{R}{12},\frac{R}{6}\}$, and $R$ is constant. From (2.1),   we get  $\nabla W^{\pm}=0$. By Proposition 5 in \cite{D}, $(M^4, g)$ is a K\"{a}hler manifold of positive constant scalar curvature.
\end{proof}

By Remark 2.3, we can rewrite Theorem B as follows:

\noindent
{ {\bf Theorem B*.}
Let $(M^4, g)$ be a  $4$-dimensional   compact Riemannian manifold with
$\delta W^{+}=0$ and positive Yamabe constant $\mu([g])$. Then either $(M^4, g)$ is anti-self-dual, or
\begin{equation}\int_{M}|W_g^{\pm}|^{2}\geq16\int_M \sigma_2(A),\end{equation}
 Furthermore, equality holds in (2.10) if and only if $g$ is a positive Einstein which is either K\"{a}hler, or the quotient of a K\"{a}hler manifold by a free, isometric,
anti-holomorphic involution.
}
\begin{remark}
$\int_{M}|W_g^{\pm}|^{2}\geq\frac 16\mu^2([g])$ implies that \begin{equation}\int_{M}|W_g^{\pm}|^{2}\geq\frac {16}{3}\pi^2(2\chi(M^4)\pm3\sigma(M^4)).\end{equation}
In fact, we recall the
following lower bound for the Yamabe invariant on  compact four-manifolds which was proved by M. J. Gursky (see \cite{{G}}):
\begin{equation}96\int_M \sigma_2(A)=\int_M R^2-12\int_M|\mathring{Ric}|^2\leq{\mu^2([g])},\end{equation}
where $\sigma_2(A)$ denotes the second-elementary function of the eigenvalues of the Schouten
tensor $A$, the inequality is strict unless $(M^4, g)$ is conformally Einstein. By the Chern-Gauss-Bonnet formula ( see Equation 6.31 of \cite{B})
$$\int_M|W|^{2}-2\int_M|\mathring{Ric}|^2+\frac{1}{6}\int_M R^2
=32{\pi}^2\chi(M)$$
we obtain
 \begin{equation}\int_{M}|W_g^{\pm}|^{2}\geq-2\int_M|\mathring{Ric}|^2+\frac{1}{6}\int_M R^2
=32{\pi}^2\chi(M)-\int_M|W|^{2}\end{equation}
Combining (1.1) with the above, we can prove (2.11).

(2.10) implies that \begin{equation}\int_{M}|W_g^{\pm}|^{2}=\frac {16}{3}\pi^2(2\chi(M^4)\pm3\sigma(M^4)).\end{equation}
In fact, by Theorem D, we have $\int_{M}|W_g^{\pm}|^{2}=\frac 16\mu^2([g])= 16\int_M \sigma_2(A)$. Hence from (2.13), (2.14) holds.

For four-manifolds $M^4$ with harmonic Weyl tensor and  positive Yamabe constant $\mu([g])$ which is not locally conformally flat , the
 lower bound for $\mu([g])$ is given by i) if $M$ is anti-self-dual, $\mu^2([g])\leq 6 \int_M|W^{-}|^2$;  ii) if $M$ is self-dual, $\mu^2([g])\leq 6 \int_M|W^{+}|^2$;  iii) if $M$ is neither anti-self-dual nor self-dual, $\mu^2([g])\leq 6 \min\{\int_M|W^{-}|^2, \int_M|W^{+}|^2\}$. The Yamabe constant $\mu^2([g])$ of  a compact positive K\"{a}hler-Einstein manifold $(M^4, g)$   is equal to $32\pi^2(2\chi(M^4)+3\sigma(M^4))$.
\end{remark}
\begin{proof}By Theorem E, we get
$$\int_{M}|W_g^{+}|^{2}\geq\frac 16\mu^2([g]).$$
 Hence  we get from (2.10)
\begin{equation*}\int_{M}|W_g^{+}|^{2}=\frac 16{\mu^2([g])}= 16\int_M \sigma_2(A).\end{equation*}
So  $g$ is conformal to an Einstein metric $\tilde{g}$. By Theorem 1.1, we get that $(M^4, g)$ is a K\"{a}hler manifold of positive constant scalar curvature.

Assume that $\tilde{g}=\lambda^2 g$. We now claim that $\lambda$ is constant, i.e., $g$ is Einstein metric. To see this, first notice
that $\tilde{g}$ being Einstein metric implies that $\delta W_{\tilde{g}}^{+}=0$. We recall this transformation law about $W^{+}$, i.e.,
\begin{equation}
{\delta}_{\tilde{g}}W_{\tilde{g}}^{+}={\delta}_{g} W_{g}^{+}-W_{g}^{+}(\frac{\nabla \lambda}{\lambda},\ldots,).
\end{equation}
It is easy to see from (2.15) that
 \begin{equation}
W_{g}^{+}(\frac{\nabla \lambda}{\lambda},\ldots,)=0.
\end{equation}
Now any oriented four-manifold $W^{+}$ satisfies (see \cite{D})
\begin{equation}
(W^{+})^{ikpq}(W^{+})_{jkpq}=|W^{+}|^2\delta_j^i.
\end{equation}
Pairing both sides of (2.17) with $(d\lambda\otimes d\lambda)_i^j$
 and using (2.16) we
get $|W^{+}|^2|\nabla\lambda|^2=0$. Since $|W_{\tilde{g}}^{+}|$ is constant, $W^{+}$ never vanishes, so $\nabla\lambda=0$
and $\lambda$ is constant.

Hence we conclude that $(M^4, g)$ is an Einstein manifold
which is either K\"{a}hler, or the quotient of a K\"{a}hler manifold by a free, isometric,
anti-holomorphic involution.
\end{proof}

\begin{theorem}
 Let $(M^4, g)$ be a  $4$-dimensional   compact analytic Riemannian manifold with
$\delta W^{\pm}=0$. Then
$$\int_M|W^{\pm}|^{\frac{1}{3}}[R-\sqrt{6}|W^{\pm}|]\leq 0$$
and equality occurs if and only if either $(M^4, g)$ is anti-self-dual (resp., self-dual), or  $(M^4, g)$ is conformal to a K\"{a}hler manifold.
\end{theorem}
\begin{proof}
We choose $\alpha=\frac 16$ in (2.6), obtain
\begin{equation}\int_M|W^{\pm}|^{\frac{1}{3}}[R-\sqrt{6}|W^{\pm}|]\leq 0.\end{equation}

If the equality holds in (2.18), all inequalities leading to (2.6)
become equalities. Hence at every point, either $W^{\pm}$ is null, i.e., $M$ is anti-self dual, or it has an eigenvalue of
multiplicity $2$ and another of multiplicity $1$. Since $M^4$ is analytic, either $(M^4, g)$ is anti-self-dual (resp., self-dual), or $W^{\pm}$  has exactly two distinct eigenvalues at each point.
 By Proposition 5 in \cite{D}, $(M^4, g)$ is conformal to a K\"{a}hler manifold .
\end{proof}
\section{Four manifolds with  harmonic Weyl tensor}
\begin{proof}[{\bf Proof of Theorem
1.2}]
By Theorem E, we have that  $W^{+}=0, \int_M|W^{-}|=\frac{1}{6}{\mu^2([g])}$, or $ W^{-}=0, \int_M|W^{+}|=\frac{1}{6}{\mu^2([g])}$.
By Theorem 1.1,  $(M^4, {g})$ is  a K\"{a}hler manifold of positive constant scalar curvature.

 When $W^{+}=0$, by Corollary 1 in \cite{D}, the scalar curvature of $(M^4, {g})$ is $0$, and $\mu([g])=0$. Contradiction.

  When $W^{-}=0$,  by Lemma 7 in \cite{D}, $(M^4, {g})$ is locally symmetric. By the result of Bourguignon \cite{Bo2}, $(M^4, {g})$  is Einstein.
  Then ${g}$ is both Einstein and half conformally flat. By the classification theorem
of Hitchin (see \cite{B}), $(M^4, {g})$ is isometric to either a quotient of $\mathbb{S}^4$ with the round metric or
$\mathbb{CP}^2$ with the Fubini-Study metric. Since we are assuming that is not locally conformal flat, $(M^4, {g})$ is $\mathbb{CP}^2$ with the Fubini-Study metric.
\end{proof}

\begin{corollary}
Let $(M^4, g)$  be a $4$-dimensional   complete Einstein manifold  with   positive scalar curvature. If
\begin{equation}\int_M|W|^{2}=\frac{1}{6}\mu^2([g]),\end{equation}
then  $M^4$ is a $\mathbb{CP}^2$ with the Fubini-Study metric.
\end{corollary}
\begin{remark} If ``$=$'' in (3.1) is replaced by ``$<$'',   Xiao and the author \cite{{FX3}}  proved that $M^4$ is a quotient of the round $\mathbb{S}^4$, which is  proved by Theorem E. For  dimensions $n>4$, under some $L^{\frac n2}$ pinching condition, G. Canto \cite{Ca2} and the author \cite{{FX2},{FX3}}  proved that $M$ is a quotient of the round $\mathbb{S}^n$, respectively.
\end{remark}

\begin{theorem}
Let $(M^4, g)$  be a $4$-dimensional   compact Riemannian manifold  with harmonic Weyl tensor and   positive Yamabe constant. If
 \begin{equation*}\int_M|W|^{2}+2\int_M|\mathring{Ric}|^2
\leq\frac{1}{6}\int_M R^2, \text{i.e.,} \int_M|W|^{2}\leq 16{\pi}^2\chi(M).\end{equation*}
Then  1) $M^4$ is a locally conformally flat manifold. In particular, $\tilde{g}$ is a Yamabe minimizer and $(M^4, \tilde{g})$  is  the stand sphere $\mathbb{S}^4$,  the real projective space $\mathbb{RP}^4$,    the manifold which is isometrically covered by $\mathbb{S}^1\times \mathbb{S}^{3}$ with the product metric,
or the manifold which is isometrically covered by $\mathbb{S}^1\times \mathbb{S}^{3}$ with a rotationally symmetric Derdzi\'{n}ski metric;

2)  $(M^4, {g})$
 is  a $\mathbb{CP}^2$ with the Fubini-Study metric.
\end{theorem}
\begin{proof}
By the Chern-Gauss-Bonnet formula,
we get
$$\int_M|W|^{2}+2\int_M|\mathring{Ric}|^2
-\frac{1}{6}\int_M R^2=2\int_M|W|^{2}
-32{\pi}^2\chi(M)\leq 0, \text{i.e.,} \int_M|W|^{2}
\leq16{\pi}^2\chi(M).$$

From (2.12), we get
$$\int_M|W|^{2}
-\frac{1}{6}{\mu^2([g])}\leq\int_M|W|^{2}+2\int_M|\mathring{Ric}|^2
-\frac{1}{6}\int_M R^2.$$
Moreover, the  inequality is strict unless $(M^4, g)$ is conformally Einstein.

In the first case ``$<$'', we have $$\int_M|W|^{2}
<\frac{1}{6}{\mu^2([g])}, \text{i.e.,} \int_M|W^{\pm}|^{2}
<\frac{1}{6}{\mu^2([g])}.$$
By Theorem E, we obtain that $M^4$ is conformally flat. Since $\int_M|W|^{2}$, ${\mu^2([g])}$ and $\int_M \sigma_2(A)$ are conformally invariant, there exists a conformal metric $\tilde{g}$ of $g$ such that ${\mu^2([g])}=\int_M R_{\tilde{g}}^2$, and
\begin{equation}\int_M|W_{\tilde{g}}|^{2}+2\int_M|\mathring{Ric}_{\tilde{g}}|^2
-\frac{1}{6}\int_M R_{\tilde{g}}^2=\int_M|W|^{2}+2\int_M|\mathring{Ric}|^2
-\frac{1}{6}\int_M R^2\leq0,\end{equation}
i.e., $$2\int_M|\mathring{Ric}_{\tilde{g}}|^2
-\frac{1}{6}\mu^2([g])\leq0.$$ By Theorems 1.13 and 1.15 in \cite{FX}, $(M^4, \tilde{g})$  is isometric to the round $\mathbb{S}^4$,  the real projective space $\mathbb{RP}^4$,    a manifold which is isometrically covered by $\mathbb{S}^1\times \mathbb{S}^{3}$ with the product metric, or a manifold which is isometrically covered by $\mathbb{S}^1\times \mathbb{S}^{3}$ with a rotationally symmetric Derdzi\'{n}ski metric.

In the second case ``$=$',  we have
$$\int_M|W|^{2}
=\frac{1}{6}{\mu^2([g])}.$$ Here $g$ is conformally
Einstein. By Theorem 1.2, $(M^4, g)$ is a $\mathbb{CP}^2$ with the Fubini-Study metric.
\end{proof}
\begin{remark}
Any compact conformally flat $4$-manifold with $\mu([g])>0$ and $\chi(M)\geq0$ has been classified \cite{{G},{G2}}. Gursky proved that $M^4$  is conformal to the round $\mathbb{S}^4$,  the real projective space $\mathbb{RP}^4$, or    a quotient of $\mathbb{R}^1\times \mathbb{S}^{3}$ with the product metric in \cite{{G},{G2}}.
Comparing  with Theorem D, it is easy to see that the condition in Theorem 3.3 is strong and  the conclusion in Theorem 3.3 is also strong.
\end{remark}

\begin{proof}[{\bf Proof of Theorem
1.4}]
 By the Chern-Gauss-Bonnet formula,
we get
\begin{equation}\int_M|W|^{2}+4\int_M|\mathring{Ric}|^2
-\frac{1}{3}\int_M R^2= 3\int_M|W|^{2}
-64{\pi}^2\chi(M)<0, \text{i.e.,} \int_M|W|^{2}
<\frac{64}{3}{\pi}^2\chi(M).\end{equation}

From (2.12), we get
\begin{equation*}\int_M|W|^{2}
-\frac{1}{3}{\mu^2([g])}\leq\int_M|W|^{2}+4\int_M|\mathring{Ric}|^2
-\frac{1}{3}\int_M R^2.\end{equation*}
Moreover, the  inequality is strict unless $(M^4, g)$ is conformally Einstein.
We have $$\int_M|W|^{2}
<\frac{1}{3}{\mu^2([g])}.$$

 a) $W=0$. Since $\int_M|W|^{2}$, ${\mu^2([g])}$ and $\int_M \sigma_2(A)$ are conformally invariant, there exists a conformally metric $\tilde{g}$ of $g$ such that ${\mu^2([g])}=\int_M R_{\tilde{g}}^2$, and from (3.3) we have $$4\int_M|\mathring{Ric}_{\tilde{g}}|^2
-\frac{1}{3}\mu^2([g])<0.$$ By Theorem 1.13  in \cite{FX}, $(M^4, \tilde{g})$  is the round $\mathbb{S}^4$ or the real projective space $\mathbb{RP}^4$.

 b) $W\neq 0$. By Theorem E, $b_1(M)=0$. Hence $\chi(M)=2+b_2$. By Theorem 3.3, we assume $16{\pi}^2\chi(M)<\int_M|W|^{2}$. By this fact that $\mu^2([g])\leq\mu^2(\mathbb{S}^4)=384\pi^2$ and the  inequality is strict unless $(M^4, g)$ is conformal to $\mathbb{S}^4$, $\int_M|W|^{2}
\leq\frac{1}{3}{\mu^2([g])}$ implies that $\chi(M)\leq 7$.
By Theorem E, we have that  $W^{+}=0$ and $\int_M|W^{-}|^2\geq\frac{1}{6}{\mu^2([g])}$, or $W^{-}=0$ and $\int_M|W^{+}|^2\geq\frac{1}{6}{\mu^2([g])}$. By Theorem F and the Hirzebruch signature formula, $b_2(M)=b^{-}_2(M)\neq0$ or $b_2(M)=b^{+}_2(M)\neq0$. Hence $3\leq\chi(M)=2+b_2\leq7$.

When $W^{-}=0$ and $\int_M|W^{+}|^2\geq\frac{1}{6}{\mu^2([g])}$,  $\chi(M)=2+b_2(M)=2+b^{+}_2(M)\leq7$. By the Hirzebruch signature formula
\begin{equation*}\frac{4\chi(M)}{9}>\frac{1}{48\pi^2}\int_{M}|W_g^{+}|^{2}=b_2^{+},\end{equation*}
  only the case  $b_2^{+}=1$ occurs.
Since $\int_{M}|W_g^{+}|^{2}=48\pi^2=\frac{16\pi^2}{3}(2\chi(M)+3\sigma(M))$, by Remark 2.3,  $\int_{M}|W|^{2}=\int_{M}|W_g^{+}|^{2}=\frac{\mu^2([g])}{6}$. Hence by Theorem 1.2, $(M^4, {g})$  is a $\mathbb{CP}^2$ with the Fubini-Study metric.

 When
$W^{+}=0$ and $\int_M|W^{-}|^2\geq\frac{1}{6}{\mu^2([g])}$. Similarly, we obtain $\int_{M}|W|^{2}=\int_{M}|W_g^{-}|^{2}=\frac{\mu^2([g])}{6}$. From the proof of Theorem 1.2, it can't happen.
\end{proof}

\begin{proof}[{\bf Proof of Theorem
1.5}]
i) When $\chi(M)=0$.

This pinching condition implies $W=0$.
From (3.2), there exists a conformally metric $\tilde{g}$ of $g$ such that ${\mu^2([g])}=\int_M R_{\tilde{g}}^2$ and $$2\int_M|\mathring{Ric}_{\tilde{g}}|^2
-\frac{1}{6}\mu^2([g])=0.$$ By Theorem  1.15 in \cite{FX}, $(M^4, \tilde{g})$  is     a manifold which is isometrically covered by $\mathbb{S}^1\times \mathbb{S}^{3}$ with the product metric, or a manifold which is isometrically covered by $\mathbb{S}^1\times \mathbb{S}^{3}$ with a rotationally symmetric Derdzi\'{n}ski metric.

ii) When $\chi(M)\neq0$. Since $\int_M|W|^{2}\leq\frac{1}{3}{\mu^2([g])}$, $\int_M|W|^{2}
=\frac{64}{3}{\pi}^2\chi(M)$ implies that $\chi(M)\leq 5$.
Since $\int_M|W|^{2}
=\frac{64}{3}{\pi}^2\chi(M)$, by (3.3) and Theorem F, $b_1(M)=0$. Hence $\chi(M)=2+b_2$.

Case 1. In the first case ``$<$'', we have $$\int_M|W|^{2}
<\frac{1}{3}{\mu^2([g])}.$$
From the proof of Theorem 1.4, we have $W^{\mp}=0$ and $\int_M|W^{\pm}|^2\geq\frac{1}{6}{\mu^2([g])}$, and $3\leq\chi(M)=2+b_2(M)=2+b^{+}_2(M)\leq5$. By the Hirzebruch signature formula
\begin{equation*}\frac{4\chi(M)}{9}=\frac{1}{48\pi^2}\int_{M}|W_g^{\pm}|^{2}=b_2^{\pm},\end{equation*} we get  $b_2^{\pm}$ is not integer. Hence there  exists
no such manifold.

Case 2. In the second case ``$=$',  we have
$$\int_M|W|^{2}
=\frac{1}{3}{\mu^2([g])}.$$Here $g$ is conformal to a
Einstein metric.  Since $(M^4, {g})$ has harmonic Weyl tensor, from the proof of Theorem B*, we get that $(M^4, {g})$  is also Einstein. By corollary 1.9, $(M^4, {g})$  is a quotient of $\mathbb{S}^2\times \mathbb{S}^{2}$ with the product metric.
\end{proof}
\begin{theorem}
Let $(M^4, g)$  be a $4$-dimensional   compact Riemannian manifold  with harmonic Weyl tensor and   positive Yamabe constant. If
 \begin{equation*}\frac{1}{6}\mu^2([g])\leq\int_M|W|^{2}\leq\frac{1}{3}\mu^2([g]),\end{equation*}
then
 1) $(M^4, {g})$  is  self-dual, but is not anti-self-dual, which has either even  $\chi(M^4)\leq 4$ and $b_2^{+}=2$ or odd $\chi(M^4)\leq 1$ and $b_2^{+}=1$;

2) $(M^4, {g})$  is  anti-self-dual, but is not self-dual, which has either even  $\chi(M^4)\leq 4$ and $b_2^{-}=2$ or odd $\chi(M^4)\leq 1$ and $b_2^{-}=1$;

3) $(M^4, {g})$  is a $\mathbb{CP}^2$ with the Fubini-Study metric;

4)  $(M^4, g)$ is a quotient of  a quotient of $\mathbb{S}^2\times \mathbb{S}^{2}$ with the product metric.
\end{theorem}
\begin{proof}
By Theorem E, we get that $W^{-}=0$, $W^{+}=0$ or $\int_{M}|W^{\pm}|^2=\frac{1}{6}\mu^2([g])$.

When $W^{\mp}=0$, $\int_{M}|W^{\pm}|^2\geq\frac{1}{6}\mu^2([g])$. By Theorem E , we have $b_2^{\mp}=0$.  By the Hirzebruch signature formula
\begin{equation}\pm\frac{1}{48\pi^2}\int_{M}|W_g^{\pm}|^{2}=\frac{1}{48\pi^2}\int_{M}(|W_g^{+}|^{2}-|W_g^{-}|^{2})=b_2^{+}-b_2^{-}=\pm b_2^{\pm}=\sigma(M),\end{equation} we get  $\pm\sigma(M)=b_2^{\pm}\geq1$. Since $\int_M|W|^{2}\leq\frac{1}{3}\mu^2([g])$, by this fact that $\mu^2([g])\leq\mu^2(\mathbb{S}^4)=384\pi^2$ and the  inequality is strict unless $(M^4, g)$ is conformal to $\mathbb{S}^4$, we get $b_2^{\pm}\leq2$. Then we get $\chi(M)\leq 4$.

If $\chi(M)=3$, then $b_2^{\pm}=1$ and $b_1=0$.
By Remark 2.3, we have $$\int_{M}|W^{\pm}|^2\geq\frac {16}{3}\pi^2(2\chi(M^4)\pm3\sigma(M^4)).$$ Combining with (3.4), we have
$$48\pi^2=\pm48\pi^2\sigma(M^4)=\int_{M}|W^{\pm}|^2\geq\frac {16}{3}\pi^2(2\chi(M^4)\pm3\sigma(M^4))=48\pi^2.$$
By Theorem B and Remark 2.3, $\int_{M}|W|^2=\frac 16\mu^2([g])$. By Theorem 1.2,  $(M^4, {g})$  is a $\mathbb{CP}^2$ with the Fubini-Study metric.

When $\int_{M}|W^{\pm}|^2=\frac{1}{6}\mu^2([g])$, by Theorem 1.1 $(M^4, {g})$  is a K\"{a}hler manifold of positive constant scalar curvature, and  the Weyl tensor is  parallel. Since $(M^4, {g})$  is a K\"{a}hler manifold with harmonic Weyl tensor, by Proposition 1 in \cite{D}, the Ricci tensor is parallel. Hence $\nabla {Rm}=0$, i.e., $M$ is locally symmetric.  From (2.4), by the maximum principle we get $|W^{\pm}|^2=\frac{R^2}{6}$, and $W^{\pm}$ has eigenvalues $\{-\frac{R}{12},-\frac{R}{12},\frac{R}{6}\}$.  Thus $Rm$ has eigenvalues $\{0,0,1,0,0,1\}$.  By the classification of four-dimensional symmetric spaces, it is isometric to a quotient of $\mathbb{S}^2\times \mathbb{S}^{2}$ with the product metric.
\end{proof}

\section{Four manifolds with  harmonic curvature}
\begin{theorem}
Let $(M^4, g)$  be a   $4$-dimensional   compact Riemannian manifold  with harmonic curvature and   positive  scalar curvature. If
\begin{equation}\int_M|W|^{2}+2\int_M|\mathring{Ric}|^2
=\frac{1}{6}\mu^2([g]),\end{equation}
then 1) $M^4$ is a $\mathbb{CP}^2$ with the Fubini-Study metric;

2) $M^4$ is covered isometrically by $\mathbb{S}^1\times \mathbb{S}^{3}$ with the product metric;

3) $M^4$ is covered isometrically by $\mathbb{S}^1\times \mathbb{S}^{3}$ with a rotationally symmetric Derdzi\'{n}ski metric.
\end{theorem}
\begin{proof} Case 1. $\mathring{Ric}=0$, i.e., $M$ is Einstein. By corollary 3.1, we complete the proof of Theorem 4.1.

Case 2. $\mathring{Ric}\neq0$. It is easy to see from (4.1) that $\int_{M}|W_g^{\pm}|^{2}<\frac 16\mu^2([g])$. By Theorem E, we have $W^{\pm}=0$, thus $W=0$, i.e., $M$ is locally conformally flat and $\int_M|\mathring{Ric}|^2=\frac {1}{12}\mu^2([g])$,
By Theorem 1.15 in \cite{FX}, (4.1) implies that $(M^4, g)$ is 2) or 3).
\end{proof}

\begin{theorem}
Let $(M^4, g)$  be a   $4$-dimensional   compact Riemannian manifold  with harmonic curvature and   positive  scalar curvature. If
\begin{equation}\int_M|W|^{2}+2\int_M|\mathring{Ric}|^2
<\frac{1}{6}\mu^2([g]),\end{equation}
then $M^4$ is a quotient of the round $\mathbb{S}^4$.
\end{theorem}
\begin{proof} If $\mathring{Ric}\neq0$. It is easy to see from (4.2) that $\int_{M}|W_g^{\pm}|^{2}<\frac 16\mu^2([g])$. By Theorem E, we have $W^{\pm}=0$, thus $W=0$, i.e., $M$ is locally conformally flat and $\int_M|\mathring{Ric}|^2<\frac {1}{12}\mu^2([g])$, By Theorem 1.13 in \cite{FX}, (4.2) implies that $(M^4, g)$ is a quotient of the round $\mathbb{S}^4$. Contradiction.

(4.2) implies that $\mathring{Ric}=0$, i.e., $M$ is Einstein, and $\int_{M}|W_g^{\pm}|^{2}<\frac 16\mu^2([g])$. By Theorem E, $W_g^{\pm}=0$, i.e., $M$ is locally conformally flat. Hence $M^4$ is constant curvature space. Since the Yamabe constant is positive, $M^4$ is  a quotient of the round $\mathbb{S}^4$.
\end{proof}

\begin{corollary}
Let $(M^4, g)$  be a $4$-dimensional   compact Riemannian manifold  with harmonic curvature and   positive  scalar curvature. If
 \begin{equation*}\int_M|W|^{2}+4\int_M|\mathring{Ric}|^2
\leq\frac{1}{6}\int_M R^2, \text{i.e.,} \int_M|W|^{2}+\int_M|\mathring{Ric}|^2\leq 16{\pi}^2\chi(M),\end{equation*} then  1) $M^4$ is a quotient of the round $\mathbb{S}^4$;

2) $M^4$ is a $\mathbb{CP}^2$ with the Fubini-Study metric.
\end{corollary}

\begin{proof}
By the Chern-Gauss-Bonnet formula,
we have
\begin{equation*}
\begin{split}\int_M|W|^{2}+4\int_M|\mathring{Ric}|^2
-\frac{1}{6}\int_M R^2=2\int_M|W|^{2}+2\int_M|\mathring{Ric}|^2
-32{\pi}^2\chi(M)\leq 0,\\ \text{i.e.,} \int_M|W|^{2}+\int_M|\mathring{Ric}|^2
\leq16{\pi}^2\chi(M).
\end{split}
\end{equation*}

From (2.12), we get
$$\int_M|W|^{2}+2\int_M|\mathring{Ric}|^2
-\frac{1}{6}{\mu^2([g])}\leq\int_M|W|^{2}+4\int_M|\mathring{Ric}|^2
-\frac{1}{6}\int_M R^2.$$
Moreover, the  inequality is strict unless $(M^4, g)$ is conformally Einstein.

In the first case ``$<$'',
Theorem 4.2 immediately implies $M^4$ is  a quotient of the round $\mathbb{S}^4$.

In the second case ``$=$',  we have that $g$ is conformally Einstein and
$$\int_M|W|^{2}+2\int_M|\mathring{Ric}|^2
=\frac{1}{6}{\mu^2([g])}.$$ Since $g$ has constant scalar curvature, $g$ is Einstein from the proof of Obata Theorem. By corollary 3.1, $(M^4, g)$ is $\mathbb{CP}^2$ with the Fubini-Study metric.
\end{proof}

\begin{proof}[{\bf Proof of Theorem
1.7}]
 It is easy to see from (1.3) that $\int_{M}|W_g^{+}|^{2}<\frac 16\mu^2([g])$, $\int_{M}|W_g^{\pm}|^{2}=\frac 16\mu^2([g])$  or $\int_{M}|W_g^{-}|^{2}<\frac 16\mu^2([g])$.

 When $\int_{M}|W_g^{+}|^{2}<\frac 16\mu^2([g])$, by Theorem E, we have $W^{+}=0$. Since $\frac{R}{6}I-W^{+}=\frac{R}{6}I>0$, by Theorem 4.3 of \cite{MW}, (a), (c)and (d) in Theorem 4.3 of \cite{MW} occur. If $(M^4, {g})$  is a anti-self dual K\"{a}hler manifold, by Corollary 1 in \cite{D}, the scalar curvature of $(M^4, {g})$ is $0$. Hence only (a)  in Theorem 4.3 of \cite{MW} can  occur, i.e., $M$ is conformal flat. By Theorem 1.15 in \cite{FX}, (1.3) implies that $(M^4, g)$ is 2) or 3).

When $\int_{M}|W_g^{-}|^{2}<\frac 16\mu^2([g])$, by Theorem E, we have $W^{-}=0$. By the result of Bourguignon, $g$ is conformal flat or Einstein.
If $g$ is conformal flat, $(M^4, g)$ is 2) or 3).
If $g$ is Einstein, $g$ is both Einstein and half conformally flat. By the classification theorem
of Hitchin, $(M^4, g)$ is isometric to either $\mathbb{S}^4$ with the round metric or
$\mathbb{CP}^2$ with the Fubini-Study metric $g$. Since we are assuming that is not locally conformal flat, $(M^4, g)$ is $\mathbb{CP}^2$ with the Fubini-Study metric. Since $\mathbb{CP}^2$ is Einstein, (1.3) implies that $\int_{M}|W_g^{\mp}|^{2}=\frac 13\mu^2([g])$.
On the other hand,  $\mu([g_{\mathbb{CP}^2}])=\left(\int_{\mathbb{CP}^2} R^2\right)^{\frac12}=12\sqrt{2}\pi$. By the Chern-Gauss-Bonnet formula and $\chi(\mathbb{CP}^2)=3$, we have $\left(\int_{\mathbb{CP}^2} |W|^2\right)^{\frac12}=4\sqrt{3}\pi=\frac 16\mu^2([g])$. Contradiction.

When $\int_{M}|W_g^{\pm}|^{2}=\frac 16\mu^2([g])$, from the proof of  Theorem 3.5, $M^4$ is isometric to a quotient of $\mathbb{S}^2\times \mathbb{S}^{2}$ with the product metric.
\end{proof}

\begin{proof}[{\bf Proof of Theorem
1.8}]By Theorem 1.4 in \cite{F},
 $M^4$ is Einstein. It is easy to see from (1.4) that $\int_{M}|W_g^{+}|^{2}<\frac 16\mu^2([g])$ or $\int_{M}|W_g^{-}|^{2}<\frac 16\mu^2([g])$.
By Theorem E, we have $W^{+}=0$ or $W^{-}=0$.
 By the classification theorem
of Hitchin, $(M^4, g)$ is isometric to either a quotient of $\mathbb{S}^4$ with the round metric or
$\mathbb{CP}^2$ with the Fubini-Study metric.
\end{proof}

\begin{corollary}
Let $(M^4, g)$  be a $4$-dimensional   compact  Riemannian manifold  with harmonic curvature and   positive  scalar curvature. If
\begin{equation}\int_M|W|^{2}+8\int_M|\mathring{Ric}|^2
\leq\frac{1}{3}\int_M R^2,\end{equation}
then 1)$M^4$ is isometric to a quotient of the round $\mathbb{S}^4$;

2) $M^4$ is a quotient of $\mathbb{S}^2\times \mathbb{S}^{2}$ with the product metric;

3) $M^4$ is a $\mathbb{CP}^2$ with the Fubini-Study metric.
\end{corollary}
\begin{remark}The pinching condition (4.3)  in Corollary 4.4 is equivalent to the following
\begin{equation*}\int_M|W|^{2}+\frac {1}{15}\int_M R^2
\leq\frac{128}{5}{\pi}^2\chi(M).\end{equation*}
\end{remark}

\begin{proof}[{\bf Proofs of corollary
4.4 and Remark 4.5}]From (2.12), we get
$$\int_M|W|^{2}+4\int_M|\mathring{Ric}|^2
-\frac{1}{3}{\mu^2([g])}\leq\int_M|W|^{2}+8\int_M|\mathring{Ric}|^2
-\frac{1}{3}\int_M R^2$$
Moreover, the  inequality is strict unless $(M^4, g)$ is conformally Einstein.

In the first case ``$<$'',
Theorem 1.8 immediately implies Corollary 4.4.

In the second case ``$=$'',  we have that $g$ is conformally Einstein and
$$\int_M|W|^{2}+4\int_M|\mathring{Ric}|^2
=\frac{1}{3}{\mu^2([g])}.$$ Since $g$ has constant scalar curvature, $g$ is Einstein from the proof of Obata Theorem. By Theorem 1.7, we complete the proof of this corollary.

By the Chern-Gauss-Bonnet formula,
the right-hand sides of the above can be
written as
$$\int_M|W|^{2}+8\int_M|\mathring{Ric}|^2
-\frac{1}{3}\int_M R^2=5\int_M|W|^{2}+\frac {1}{3}\int_M R^2
-128{\pi}^2\chi(M).$$
This proves Remark 4.5.
\end{proof}
\begin{theorem}
Let $(M^4, g)$  be a $4$-dimensional   compact Riemannian manifold  with harmonic curvature and   positive scalar curvature. If
 \begin{equation*}\frac{1}{6}\mu^2([g])\leq\int_M|W|^{2}\leq\frac{1}{3}\mu^2([g]),\end{equation*}
then
1) $(M^4, {g})$  is a $\mathbb{CP}^2$ with the Fubini-Study metric;

2) $(M^4, g)$ is isometric to a quotient of
$\mathbb{S}^2\times \mathbb{S}^{2}$ with the product metric.
\end{theorem}
\begin{proof}By Theorem 3.5, We just need to consider that $(M^4, {g})$  is  self-dual or anti-self-dual.

When $(M^4, {g})$  is  self-dual. Since $(M^4, g)$ has harmonic curvature, $(M^4, g)$ is analytic \cite{DG}. By Proposition 7 in \cite{D}, we get that $(M^4, {g})$  is Einstein. By the classification theorem
of Hitchin, $(M^4, g)$ is isometric to a
$\mathbb{CP}^2$ with the Fubini-Study metric $g$.

When $(M^4, {g})$  is  anti-self-dual,  $\frac{R}{6}I-W^{+}=\frac{R}{6}I>0$. Since $(M^4, {g})$ is not self-dual, by Theorem 4.3 of \cite{MW}, only (c)and (d) in Theorem 4.3 of \cite{MW} occur, i.e., $(M^4, {g})$  is a K\"{a}hler manifold of positive constant scalar curvature. By Corollary 1 in \cite{D}, the scalar curvature of $(M^4, {g})$ is $0$. Contradiction.
\end{proof}

By Theorems E, 1.1 and 4.6, we have Theorem 1.11.

\section{Four manifolds with  positive Yamabe constant}
Chang, Gursky and Yang's proof of Theorem D is based on establishing the existence of a solution of a forth order fully nonlinear
equation. By avoiding the existence of a solution of a fourth order fully nonlinear equation, we can reprove Theorem D.
\begin{theorem}
Let $(M^4, g)$  be a $4$-dimensional   compact Riemannian manifold  with   positive Yamabe constant. If
 \begin{equation*}\int_M|W|^{2}\leq 16{\pi}^2\chi(M),\end{equation*} then  1) $\tilde{g}$ is a Yamabe minimizer and $(M^4, \tilde{g})$  is      the manifold which is isometrically covered by $\mathbb{S}^1\times \mathbb{S}^{3}$ with the product metric,
or  $\mathbb{S}^1\times \mathbb{S}^{3}$ with a rotationally symmetric Derdzi\'{n}ski metric;

2)  $M^4$ is diffeomorphic to the stand sphere $\mathbb{S}^4$ or  the real projective space $\mathbb{RP}^4$;

3)  $\tilde{g}$ is a Yamabe minimizer and $(M^4, \tilde{g})$  is a $\mathbb{CP}^2$ with the Fubini-Study metric.
\end{theorem}

\begin{proof} i) When $\chi(M)=0$.

This pinching condition implies $W=0$.
Since $\int_M|W|^{2}$, ${\mu^2([g])}$ and $\int_M \sigma_2(A)$ are conformally invariant, there exists a conformally metric $\tilde{g}$ of $g$ such that ${\mu^2([g])}=\int_M R_{\tilde{g}}^2$, and $$2\int_M|\mathring{Ric}_{\tilde{g}}|^2
-\frac{1}{6}\mu^2([g])\leq0.$$ Since $\chi(M)=0$, by Theorems 1.13 and 1.15 in \cite{FX}, $(M^4, \tilde{g})$  is     a manifold which is isometrically covered by $\mathbb{S}^1\times \mathbb{S}^{3}$ with the product metric, or a manifold which is isometrically covered by $\mathbb{S}^1\times \mathbb{S}^{3}$ with a rotationally symmetric Derdzi\'{n}ski metric.

ii) When $\chi(M)\neq0$.

Case 1. In the first case ``$<$'', we have $$\int_M|W|^{2}
<\frac{1}{6}{\mu^2([g])}.$$ By Theorem E, $b_2(M)=0$.

a) $W=0$. Since $\chi(M)\neq0$, by Theorems 1.13 and 1.15 in \cite{FX}, $(M^4, \tilde{g})$  is the round $\mathbb{S}^4$,  the real projective space $\mathbb{RP}^4$.

b) $W\neq 0$. Since $\int_M|W|^{2}$, ${\mu^2([g])}$ and $\int_M \sigma_2(A)$ are conformally invariant, and $W\neq0$, there exists a conformally metric $\tilde{g}$ of $g$ such that ${\mu^2([g])}=\int_M R_{\tilde{g}}^2$, and $$2\int_M|\mathring{Ric}_{\tilde{g}}|^2
-\frac{1}{6}\mu^2([g])<0.$$
 By Theorem E, $b_1(M)=0$. Hence $M^4$ is a finite quotient of the homological sphere. By Freedman's result \cite{Fr}, $M^4$ is covered by a homeomorphism sphere. By the Hirzebruch signature formula,  $$\int_M|W^{+}|^{2}=\int_M|W^{-}|^{2}=\frac12\int_M|W|^{2}<\frac{1}{12}\mu^2([g])$$ Thus by Section 2 in \cite{CZ} and Section 3 in \cite{G3} , we get
$$|W_{\tilde{g}}^{+}|^{2}<\frac{1}{12}R_{\tilde{g}}^{2}, |W_{\tilde{g}}^{-}|^{2}<\frac{1}{12} R_{\tilde{g}}^2.
$$
Let $\lambda^\pm_1\geq\lambda\pm_2\geq\lambda\pm_3$ be the eigenvalues of $W^{\pm}$. Since $W^{\pm}$ is of trace free, we have $\lambda^\pm_1+\lambda^\pm_2+\lambda^\pm_3=0$, and
$$\frac32{\lambda^\pm_1}^2\leq{\lambda^\pm_1}^2+\frac12(\lambda^\pm_2+\lambda^\pm_3)^2={\lambda^\pm_1}^2+{\lambda^\pm_2}^2+{\lambda^\pm_3}^2
=\frac14|W^{\pm}|^{2}<\frac{1}{48}R^{2},$$
i.e., $\lambda^\pm_1<\frac{\sqrt{2}}{12}R$. Hence $\lambda^\pm_2+\lambda^\pm_3>-\frac{\sqrt{2}}{12}R$. This implies the sum of least two eigenvalue of $\frac{1}{12}R+W^{\pm}$ is positive. So $(M^4, \tilde{g})$ has positive isotropic curvature. Since $M^4$ is covered by a homeomorphism sphere, by the main theorem of \cite{CTZ},  $M^4$ is diffeomorphic to the stand sphere $\mathbb{S}^4$ or  the real projective space $\mathbb{RP}^4$;

Case 2. In the second case ``$=$',  we have
$$\int_M|W|^{2}
=\frac{1}{6}{\mu^2([g])}=16\pi^2\chi(M).$$Hence $g$ is conformal to a
Einstein metric $\tilde{g}$.
 By corollary 3.1, $(M^4, {g})$  is conformal to a $\mathbb{CP}^2$ with the Fubini-Study metric.
\end{proof}

\begin{remark}
The proof of Chang-Gursky-Yang consists of two steps. First, they prove
this case $``<"$; Second, based on the first step, they prove this case $``="$. We unify the
two cases.
Chen and Zhu \cite{CZ} prove a classification theorem of $4$-manifolds which generalizes Theorem C under the strict inequality assumption.
\end{remark}

Based on the first Weitzenb\"{o}ck formulas in Remark 2.2, using the same argument as in the proof of Theorem 1.1, we can obtain the following result of Gursky \cite{G3}.

\noindent
{{\bf Theorem G (Gursky \cite{G3}).} Let $(M^4, g)$ be a  $4$-dimensional   compact Riemannian manifold with
 positive Yamabe constant $\mu([g])$. If $b^\pm_2\neq0$ and
\begin{equation}\int_{M}|W_g^{\pm}|^{2}=\frac 16\mu^2([g]).\end{equation}
then $(M^4, g)$ is conformal to  a K\"{a}hler manifold of positive constant scalar curvature.
}
\begin{proof} Since  $b^\pm_2\neq0$, there exists a nonzero  $\omega^{\pm}\in H^2_{\pm}(M).$ Setting $u=|\omega^{\pm}|$. Based on the first Weitzenb\"{o}ck formulas in Remark 2.2, using the same argument as in the proof of Theorem 1.1, we get
\begin{equation}
\begin{split}
0\geq(2-\frac{1}{2\alpha})\frac16\mu([g])\left(\int_M  u^{4\alpha}\right)^{\frac{1}{2}}-\frac{\sqrt{6}}{3}\alpha\left(\int_{M}u^{4\alpha}\right)^{\frac{1}{2}}\left(\int_{M}|W^{\pm}|^{2}\right)^{\frac{1}{2}}\\
+\frac {(2\alpha-1)^2}{12\alpha}\int_M Ru^{2\alpha}.
\end{split}
\end{equation}
We choose $\alpha=\frac 12$, from (5.2) we get
\begin{equation}
\begin{split}
0\geq\left[\frac{1}{\sqrt{6}}\mu([g])-\left(\int_{M}|W^{\pm}|^{2}\right)^{\frac{1}{2}}\right]\left(\int_M  u^{2}\right)^{\frac{1}{2}}.
\end{split}
\end{equation}
(5.1) implies that the equality holds in (5.3).
When the equality holds in (5.3), all inequalities leading to (5.2)
become equalities. From (5.2), the function $u^\alpha$ attains the infimum in the Yamabe functional. Hence the metric $\tilde{g}=u^{2\alpha}$ is a Yamabe minimizer. Then we get $|\omega^{\pm}|_{\tilde{g}}=1$. Since $\int_{M}|W^{\pm}|^{2}$ is conformally invariant,  the equality for the H\"{o}lder inequality implies that  $|W^{\pm}|_{\tilde{g}}$ is constant. From (5.1), we get $|W^{\pm}|_{\tilde{g}}^{2}=\frac 16 R_{\tilde{g}}^2$.
By the first  Weitzenb\"{o}ck formula and the maximum principle, we get $|\omega|$ is constant, thus $\nabla \omega=0$, i.e., $(M^4, \tilde{g})$ is a K\"{a}hler manifold of positive constant scalar curvature. Hence $(M^4, g)$ is conformal to  a K\"{a}hler manifold of positive constant scalar curvature.
\end{proof}

Based on Theorems F and G, using the same arguments as in the proof of Theorem B*, we can reprove Theorem A which is rewritten as follows:

\noindent
{ {\bf Theorem A*.}
Let $(M^4, g)$ be a  $4$-dimensional   compact Riemannian manifold with
 positive Yamabe constant $\mu([g])$ and $H^2_{+}(M)\neq0$.  Then
\begin{equation*}\int_{M}|W_g^{+}|^{2}\geq16\int_M \sigma_2(A).\end{equation*}
 Furthermore, equality holds in the above inequality  if and only if $g$ is conformal to a positive K\"{a}hler-Einstein metric.
}

Based on Theorems F,  we can reprove Theorem C which is rewritten as follows:

\noindent
{ {\bf Theorem C*.} Let $(M^4, g)$  be a $4$-dimensional   compact Riemannian manifold  with  positive Yamabe constant  and the space of harmonic $1$-forms $H^1(M^4)\neq0$.
Then
 \begin{equation*} \int_M|W^{+}|^{2}=8{\pi}^2(2\chi(M^4)+3\sigma(M^4))-8\int_M \sigma_2(A)\geq 8{\pi}^2(2\chi(M^4)+3\sigma(M^4)).\end{equation*}
 Furthermore, the equalities holds in the above inequalities if and only if
$(M^4, g)$ is conformal to a quotient of $\mathbb{R}^1\times \mathbb{S}^{3}$ with the product metric.}
\begin{proof}
By Theorem F, $\int_M|\mathring{Ric}|^2\geq\frac{1}{12}\mu^2([g])$ for $H^1(M^4)\neq0$. Since $M^4$ is compact, there exists a conformally metric $\tilde{g}$ of $g$ such that $\mu^2([g])=\int_M R_{\tilde{g}}^2$. Hence we get
$$-2\int_M|\mathring{Ric}_{\tilde{g}}|^2+\frac 16\mu^2([g])=-2\int_M|\mathring{Ric}_{\tilde{g}}|^2+\frac{1}{6}\int_M R_{\tilde{g}}^2=-2\int_M|\mathring{Ric}|^2+\frac{1}{6}\int_M R^2\leq0,$$
i.e., $$16\int_M \sigma_2(A)\leq0.$$
By the Chern-Gauss-Bonnet formula,
$$\int_M|W^{+}|^{2}= 8{\pi}^2(2\chi(M^4)+3\sigma(M^4))-8\int_M \sigma_2(A).$$
Hence  \begin{equation*} \int_M|W^{+}|^{2}\geq 8{\pi}^2(2\chi(M^4)+3\sigma(M^4)).\end{equation*}
From the proof of Theorem F and the above, the equalities hold in the above inequalities imply that $|\nabla|\omega||^2=\frac34|\nabla\omega|^2$ and
$\int_M|\mathring{Ric}|^2=\frac{1}{12}\mu^2([g])=\frac{1}{12}\int_M R^2$. By Proposition 5.1 and subsection 7.2 in \cite{BC}, $(M^4, g)$ is conformal to a quotient of $\mathbb{R}^1\times \mathbb{S}^{3}$ with the product metric.
\end{proof}

\begin{proof}[{\bf Proof of Theorem
1.3}]
Case 1. $\int_M|W^{\pm}|<\frac{1}{6}{\mu^2([g])}$. By Theorem E, we get $b^{\pm}_2=0$. From the proof of Theorem 5.1, we get $(M^4, {g})$ has positive isotropic curvature. According to the main theorem in \cite{CTZ}, it is
diffeomorphic to $\mathbb{S}^4,$ $\mathbb{RP}^4$, $\mathbb{S}^3\times \mathbb{R}/G$ or a connected sum of them. Here
$G$ is a cocompact fixed point free discrete subgroup of the isometry group
of the standard metric on $\mathbb{S}^3\times \mathbb{R}$.

Case 2.  $W^{+}=0, \int_M|W^{-}|^2=\frac{1}{6}{\mu^2([g])}$, or $ W^{-}=0, \int_M|W^{+}|^2=\frac{1}{6}{\mu^2([g])}$. By the Hirzebruch signature formula
\begin{equation*}\frac{1}{48\pi^2}\int_{M}(|W_g^{+}|^{2}-|W_g^{-}|^{2})=b_2^{+}-b_2^{-}=\sigma(M),\end{equation*} we get $b_2^{-}\neq0$ or $b_2^{+}\neq0$.
By Theorem F,  $(M^4, {g})$ is conformal to a K\"{a}hler manifold of positive constant scalar curvature.

 When $W^{+}=0$, by Corollary 1 in \cite{D}, the scalar curvature of $(M^4, \tilde{g})$ is $0$, and $\mu([g])=0$. Contradiction.

  When $W^{-}=0$,  by Lemma 7 in \cite{D}, $(M^4, \tilde{g})$ is locally symmetric. By the result of Bourguignon, $(M^4, \tilde{g})$  is Einstein.
  Then $\tilde{g}$ is both Einstein and half conformally flat. By the classification theorem
of Hitchin, $(M^4, \tilde{g})$ is isometric to either a quotient of $\mathbb{S}^4$ with the round metric or
$\mathbb{CP}^2$ with the Fubini-Study metric. Since we are assuming that is not locally conformal flat, $(M^4, \tilde{g})$ is $\mathbb{CP}^2$ with the Fubini-Study metric.
\end{proof}
\begin{remark}
Although the conditions Theorem 1.6 of \cite{CZ} is weaker than the ones in Theorem 1.3, I do not know that Theorem 1.3 can  be deduced directly from Theorem 1.6 of \cite{CZ} as far as my knowledge is concerned. Chen and Zhu proved Theorem 1.6 of \cite{CZ} by using Micallef-Wang's result \cite{MW} which I do not use in the proof of Theorem 1.3. From the proof of Theorem 5.1, Theorem D can be deduced from Theorem 1.3.
\end{remark}

\begin{proof}[{\bf Proof of Theorem
1.6}] From the proof of Theorem 1.4, we have $b_1=0$ and  $2\leq\chi(M)\leq7$.

Case 1. In the first case ``$<$'', we have $$\int_M|W|^{2}
<\frac{1}{3}{\mu^2([g])}.$$

When $\int_M|W^{\pm}|^2<\frac{1}{6}{\mu^2([g])}$,  by Theorem E, $b^{+}_2(M)=b^{-}_2(M)=0$. Hence $M^4$ is covered by a homeomorphism sphere.
From the proof of Theorem 5.1, $(M^4, g)$ has positive isotropic curvature. By the main theorem of \cite{CTZ},  $M^4$ is diffeomorphic to the stand sphere $\mathbb{S}^4$ or  the real projective space $\mathbb{RP}^4$;

When $\int_M|W^{+}|^2<\frac{1}{6}{\mu^2([g])}$ and $\int_M|W^{-}|^2\geq\frac{1}{6}{\mu^2([g])}$, or $\int_M|W^{-}|^2<\frac{1}{6}{\mu^2([g])}$ and $\int_M|W^{+}|^2\geq\frac{1}{6}{\mu^2([g])}$. By Theorem E and the Hirzebruch signature formula, $b_2(M)=b^{-}_2(M)\neq0$ or $b_2(M)=b^{+}_2(M)\neq0$. Hence $3\leq\chi(M)=2+b_2\leq7$. From the proof of Theorem 1.4, we have $b^{\pm}_2=1$ and  $\chi(M)=3$. If $(M^4, g)$ has harmonic Weyl tensor, by Theorem 1.4 we have $\int_M|W|^{2}
=\frac{1}{6}{\mu^2([g])}$, which contradicts $\int_M|W|^{2}
>\frac{1}{6}{\mu^2([g])}$.

Case 2. In the second case ``$=$',  we have
$$\int_M|W|^{2}
=\frac{1}{3}{\mu^2([g])}=\frac{64}{3}{\pi}^2\chi(M).$$Hence $g$ is conformal to a
Einstein metric $\tilde{g}$.   By corollary 1.9, $(M^4, {g})$  is conformal to a quotient of $\mathbb{S}^2\times \mathbb{S}^{2}$ with the product metric.
\end{proof}

\begin{theorem}
Let $(M^4, g)$  be a $4$-dimensional   compact Riemannian manifold  with   positive Yamabe constant. If
 \begin{equation*}\frac{1}{6}\mu^2([g])\leq\int_M|W|^{2}\leq\frac{1}{3}\mu^2([g]),\end{equation*}
 and the universal cover of $(M^4, {g})$ is not diffeomorphic to a quotient of $\mathbb{S}^4$.

Then
 1) the universal cover of $(M^4, {g})$   has either  $\chi(M^4)=4$ and $b_2=b_2^{+}=2$ or $\chi(M^4)=3$ and $b_2=b_2^{+}=1$;

2) the universal cover of $(M^4, {g})$   has either   $\chi(M^4)=4$ and $b_2=b_2^{-}=2$ or  $\chi(M^4)=3$ and $b_2=b_2^{-}=1$;

3) the universal cover of $(M^4, {g})$  is a K\"{a}hler manifold of positive constant scalar curvature. In particular, $(M^4, g)$ is a quotient of  $(\Sigma_1, g_1)\times (\Sigma_2, g_2)$, where the surface $(\Sigma_i, g_i)$ has constant Gaussian curvature $k_i$, and $k_1+k_2>0$.
\end{theorem}
\begin{proof}
When $\int_{M}|W^{+}|^2<\frac{1}{6}\mu^2([g])$ and $\int_{M}|W^{-}|^2<\frac{1}{6}\mu^2([g])$. By Theorem E, we have $b_2=0$. By the proof of Case 1 in Theorem 1.3, the universal cover of $(M^4, {g})$ is diffeomorphic to a quotient of $\mathbb{S}^4$.

When $\int_{M}|W^{\mp}|^2<\frac{1}{6}\mu^2([g])$ and $\int_{M}|W^{\pm}|^2\geq\frac{1}{6}\mu^2([g])$. By Theorem E, we have $b_2^{\mp}=0$.  By the Hirzebruch signature formula
\begin{equation}\frac{1}{48\pi^2}\int_{M}(|W_g^{+}|^{2}-|W_g^{-}|^{2})=b_2^{+}-b_2^{-}=\pm b_2^{\pm}=\sigma(M),\end{equation} we get  $\pm\sigma(M)=b_2^{\pm}\geq1$. Since $\int_M|W|^{2}\leq\frac{1}{3}\mu^2([g])$, by this fact that $\mu^2([g])\leq\mu^2(\mathbb{S}^4)=384\pi^2$ and the  inequality is strict unless $(M^4, g)$ is conformal to $\mathbb{S}^4$, we get $b_2^{\pm}\leq2$. Then we get $\chi(M)\leq 4$.

When $\int_{M}|W^{+}|^2=\int_{M}|W^{-}|^2=\frac{1}{6}\mu^2([g])$. Then we get that $\chi(M)$ is even. Only (a) and (c) in  Theorem 4.10 of \cite{MW}. Moreover, if (a) in  Theorem 4.10 of \cite{MW} occurs, $(M^4, g)$ becomes positive isotropic curvature and $b_2=0$.
By the proof of Case 1 in Theorem 1.3, the universal cover of $(M^4, {g})$ is diffeomorphic to a quotient of $\mathbb{S}^4$.
if (c) in  Theorem 4.10 of \cite{MW} occurs,
the universal cover of $(M^4, {g})$  is diffeomorphic to  a K\"{a}hler manifold of positive constant scalar curvature. Since the scalar curvature is of positive,  the universal cover of $(M^4, {g})$  is diffeomorphic to  $(\Sigma_1, g_1)\times (\Sigma_2, g_2)$, where $(\Sigma_i, g_i)$ is a two-dimensional manifold, the Gaussian curvature $k_i$
of $g_i$ must be a constant and satisfies $k_1+k_2>0$.
\end{proof}

From the proof of Theorems 1.7, 1.8, 4.1 and 5.1, we obtain
\begin{theorem}
 Let $(M^4, g)$  be a $4$-dimensional   compact Riemannian manifold  with  positive Yamabe constant. If
 \begin{equation*}\int_M|W|^{2}+2\int_M|\mathring{Ric}|^2\leq\frac16\mu^2([g]),\end{equation*} then
  1) $(M^4, g)$  is a $\mathbb{CP}^2$ with the Fubini-Study metric;

2) $(M^4, g)$  is conformal to  a quotient of the round $\mathbb{S}^4$, or   a quotient of $\mathbb{R}^1\times \mathbb{S}^{3}$ with the product metric;

3) $M^4$  is diffeomorphic to a quotient of the round $\mathbb{S}^4$ which has not harmonic Weyl tensor.
\end{theorem}

\begin{theorem}
 Let $(M^4, g)$  be a $4$-dimensional   compact Riemannian manifold which is not diffeomorphic to $\mathbb{S}^4$ or $\mathbb{RP}^4$  with  positive Yamabe constant.  If
 \begin{equation*}\frac16\mu^2([g])<\int_M|W|^{2}+4\int_M|\mathring{Ric}|^2\leq\frac13\mu^2([g]),\end{equation*} then
 1) $(M^4, g)$  is a quotient of  $\mathbb{S}^2\times \mathbb{S}^{2}$ with the product metric;

2) $(M^4, g)$  has $\chi(M)=3$,  $b_1=0$ and $b_2=1$, and has not harmonic Weyl tensor.
\end{theorem}

\bibliographystyle{amsplain}

\begin{thebibliography}{10}
\bibitem{A}
T. Aubin,  \textit{Some Nonlinear Problems in Riemannian Geometry}, Springer-Verlag, Berlin
1998.

\bibitem{B}
A. L. Besse, \textit{Einstein manifolds}, Springer-Verlag, Berlin, 1987.
\bibitem{B2}
V. Bour, \textit{Fourth order curvature flows and geometric applications.} arxiv: 1012.0342.
\bibitem{BC}
V. Bour and G. Carron, \textit{Optimal integral pinching results.}  arxiv: 1203.0384.
\bibitem{Bo}
J. P. Bourguignon, \textit{The magic of Weitzenb\"{o}ck formulas.}   Progress in Nonlinear Differential Equations and Their Applications, Vol \textbf{4}, 251--271.
\bibitem{Bo2}
J. P. Bourguignon, \textit{Les vari\'{e}t\'{e}s de dimension $4$ \`{a} signature non nulle dont la courbure est harmonique sont d'Einstein.}   Invent. Math., \textbf{63} (1981), 263--286.
\bibitem{BS}
S. Brendle and R. M. Schoen, \textit{Classification of manifolds with weakly $\frac 14$-pinched
curvatures.} Acta Math. \textbf{200} (2008), 1--13.
\bibitem{BW}
C. B\"{o}hm and B. Wilking,   \textit{Manifolds with positive curvature operator are space forms.} Annals of Math.
\textbf{167} (2008), 1079--1097.

\bibitem{Ca}
G. Catino, \textit{On conformally flat manifolds with constant positive scalar curvature.}  arXiv:1408.0902v1.
\bibitem{Ca2}
G. Cation, \textit{Integral pinched shrinking Ricci solitons.} arXiv:1509.07416vl [math.DG], 2015.
\bibitem{CGY}
S. -Y. A. Chang, M. J. Gursky and P. C. Yang,  \textit{A conformally invariant sphere theorem
in four dimensions.}  Publ. Math. Inst. Hautes \'{E}tudes Sci. \textbf{98} (2003), 105--143.
\bibitem{CTZ}
B. L. Chen, S. H. Tang and X. P. Zhu,  \textit{Complete classification of compact
 four manifolds with positive isotropic curvature.}  J. Diff. Geom.  \textbf{91} (2012), 41--80.

 \bibitem{CZ}
B. L. Chen and X. P. Zhu,  \textit{A conformally invariant classification theorem
in four dimensions.}  Comm. Anal. Geom.  \textbf{91} (2014), 811--831.

\bibitem{D1}
A. Derdzi\'{n}ski, \textit{On compact Riemannian manifolds with harmonic curvature.} Math. Ann. \textbf{259} (1982), 144--152.

\bibitem{D}
A. Derdzi\'{n}ski, \textit{Self-dual K\"{a}hler manifolds and Einstein manifolds of dimension four.}
Comp. Math. \textbf{49} (1983), 405--433.
\bibitem{DG}
D. M. DeTurck and H. Goldschmidt, \textit{Regularity theorems in Riemannian geometry. ii. Harmonic curvature
and the Weyl tensor.} Forum Math. \textbf{1} (1989), 377--394.

\bibitem{Fr}
M. Freedman, \textit{The topology of four-dimensional manifolds.} J. Diff. Geom. \textbf{17} (1982), 357--453.

\bibitem{F1}
H. P. Fu, \textit{Bernstein type theorems for complete submanifolds in
space forms.} Math. Nachr. \textbf{285} (2012), 236--244.

\bibitem{F}
H. P. Fu, \textit{On compact manifolds with harmonic curvature and positive scalar curvature.} arXiv:1512.00256.

\bibitem{FL}
H. P. Fu and Z. Q. Li, \textit{The structre of complete manifolds with weighted
 Poincar\'{e} inequality and minimal hypersurfaces.} International Journal of Mathematics
\textbf{21}(2010) 1421--1428.

\bibitem{FX}
H. P. Fu and L. Q. Xiao, \textit{Some $L^p$ rigidity results for complete  manifolds with harmonic curvature.}  arXiv:1511.07094.

\bibitem{FX2}
H. P. Fu and L. Q. Xiao, \textit{Einstein  manifolds with finite  $L^p$-norm of the Weyl curvature.} Submitted to Differ.
Geom. Appl. in September 1st 2015.
\bibitem{FX3}
H. P. Fu and L. Q. Xiao, \textit{Rigidity Theorem for integral pinched shrinking Ricci solitons.} 	arXiv:1510.07121.
\bibitem{G}
M. J. Gursky, \textit{Locally conformally flat four- and six-manifolds of positive scalar curvature and positive
Euler characteristic.} Indiana Univ. Math. J. \textbf{43} (1994), 747--774.


\bibitem{G2}
M. J. Gursky,  \textit{The Weyl functional, deRham cohomology, and K\"{a}hler-Einstein metrics.} Annals of Math.  \textbf{148}
(1998), 315--337.
\bibitem{G3}
M. J. Gursky, \textit{Four-manifolds with $\delta W^{+}=0$ and Einstein constants of the sphere.} Math. Ann. \textbf{318} (2000), 417--431.


\bibitem{HV}
E. Hebey  and  M. Vaugon,  \textit{Effective $L^p$ pinching for the concircular curvature.} J. Geom. Anal. \textbf{6}
(1996),  531--553.
\bibitem{Hu}
G. Huisken,  \textit{Ricci deformation of the metric on a Riemannian manifold.} J. Differential Geom. \textbf{21} (1985),
47--62.



\bibitem{LP}
M. J. Lee and T. H. Parker,  \textit{The Yamabe problem.} Bull.
A. M. S. \textbf{17} (1987), 37--91.


\bibitem{M}
C. Margerin,  \textit{A sharp characterization of the smooth 4-sphere in curvature terms.}
 Comm. Anal. Geom. \textbf{6} (1998), 21--65.

\bibitem{MW}
M. Micallef and M. Wang,  \textit{Metrics with nonnegative isotropic curvatures.}
 Duke Math. J. \textbf{72} (1993), 649--672.
\end{thebibliography}

\end{document}